\documentclass[review]{elsarticle}

\usepackage{dsfont}
\usepackage{amsmath}
\usepackage{amsfonts, amssymb, mathrsfs}
\usepackage{color}
\usepackage[colorlinks=true,linkcolor=blue]{hyperref}

\usepackage{lineno}

\modulolinenumbers[5]

\newtheorem{theorem}{Theorem}[section]
\newtheorem{lemma}[theorem]{Lemma}

\newtheorem{definition}{Definition}[section]
\newproof{proof}{Proof}

\numberwithin{theorem}{section}

\newcommand{\Cn}{ \mathbb{C}^n }
\newcommand{\CO} {\Cn\setminus\Omega}

\newcommand{\Om}{\Omega = \left\{z\in\Cn : \rho(z)<0 \right\}}
\newcommand{\dom}{ {\partial\Omega} }

\renewcommand{\O}{\Omega}
\newcommand{\eps}{\varepsilon}

\newcommand{\C}{\mathbb{C}}

\newcommand{\N}{\mathbb{N}}

\newcommand{\RE}{\text{\rm{Re}}}
\newcommand{\IM}{\text{\rm{Im}}}
\newcommand{\pr}{ \text{\rm{pr}}}

\newcommand{\cn}{\frac{1}{(2\pi i)^n}}
\newcommand{\dbar}{\bar\partial}
\newcommand{\intl}{\int\limits}
\newcommand{\suml}{\sum\limits}

\newcommand{\dist}[2]{ \text{\rm{dist}} (#1,\ #2)}

\newcommand{\supp}[1]{ \text{\rm{supp}}\ {#1} }
\newcommand{\BMO}{ \text{\rm{BMO}}}

\newcommand{\scp}[2]{ \left\langle #1,\ #2  \right\rangle }
\newcommand{\V}[2]{\scp{\partial\rho(#1)}{#1 - #2}}
\newcommand{\norm}[1]{\left\lVert #1 \right\rVert}
\newcommand{\abs}[1]{\left\lvert #1 \right\rvert}

\begin{document}

\begin{frontmatter}

\title{External area integral inequality for the Cauchy-Leray-Fantappi\`{e} integral \footnote{The work is supported by Russian Science Foundation Grant 14-41-00010.}}

\author[mymainaddress]{Alexander Rotkevich}
\ead{rotkevichas@gmail.com}
\address[mymainaddress]{Department of Mathematical analysis, Mathematics and Mechanics Faculty, St. Petersburg State University, 198504,  Universitetsky prospekt, 28, Peterhof, St. Petersburg, Russia}

\begin{abstract}
In this paper we extend Luzin inequality for functions defined by the Cauchy-Leray-Fantappi\`{e} integral on the complement of a convex domain in $\mathbb{C}^n$. 
\end{abstract}


\end{frontmatter}

\section{Introduction\label{intro}}

Let $G\subset\C$ be a Radon domain, for $z\in \partial G$ we consider a sector $S(z)=\{\xi\in G: \dist(\xi,\partial G)\geq \frac12 |\xi-z|\}.$ It is well known (see \cite{S58}, \cite{D77}), that for function $f$ holomorphic on $G$ one has
\begin{equation}
 \norm{I_f}_{L^p(\partial G)} \leq c(p,G)\norm{f}_{L^p(\partial G)},\ 1<p<\infty,
\end{equation}
with some constant $c(p,G),$ where $I_f$ is an area-integral
\begin{equation}
 I_f(z) = \left(\int\limits_{S(z)} |f'(\xi)|^2 d\mu(\xi)\right)^{1/2}
\end{equation}
and $d\mu$ is Lebesgue measure on $\C.$

There are many many generalizations of this inequality for holomorphic functions on regular domains in $\Cn$ by P. Ahern, J. Bruna (\cite{AB88}), A. Nagel, E.M. Stein, S. Wainger \cite{NSW81}, G. Sardine \cite{Sa93},  S. Krantz and S.Y. Lee \cite{KL97}. Our main result (theorems \ref{thm:area_int}, \ref{thm:area_int_BMO}) is extension of this inequality to functions generated by Cauchy-Leray-Fantappi\`{e} integral and defined on the complement of a convex domain in~$\mathbb{C}^n$. The motivation of this paper is the possibility of applications of these results to the characterization of spaces of analytic functions by pseudoanalytical extensions (see \cite{R13}).

\section{Main notations and definitions \label{notations}}

Let $\Cn$ be the space of $n$ complex variables, $n\geq 2,$ $z =
(z_1,\ldots, z_n),\ z_j = x_j + i y_j;$
$$\partial_j f =\frac{\partial f}{\partial z_j} = \frac{1}{2}\left( \frac{\partial f}{\partial x_j} - i \frac{\partial f}{\partial y_j}\right), \quad \bar\partial_j f = \frac{\partial f}{\partial\bar{z}_j} = \frac{1}{2}\left( \frac{\partial f}{\partial x_j} + i \frac{\partial f}{\partial y_j}\right),$$

$$\partial f = \suml_{k=1}^{n} \frac{\partial f}{\partial z_k} dz_k,\quad \bar{\partial} f = \suml_{k=1}^{n} \frac{\partial f}{\partial\bar{z}_k} d\bar{z}_k,\quad df=\partial f+  \bar{\partial} f.$$

\noindent The notation
$$\scp{\partial f(z)}{w} = \suml_{k=1}^{n} \frac{\partial f(z)}{\partial z_k} w_k.$$
is used to indicate the action of $\partial f$ on the vector
$w\in\Cn,$ and $$|\bar\partial f| = \abs{\frac{\partial f}{\partial
z_1}} + \ldots+\abs{\frac{\partial f}{\partial z_n}}.$$

The euclidean distance form the point $z\in\Cn$ to the set
$D\subset\Cn$ we denote as $\dist{z}{D} = \inf\{\abs{z-w}:w\in D\}.$
Lebesgue measure in $\Cn$ we denote as $d\mu.$

For a multiindex $\alpha =
(\alpha_1,\ldots,\alpha_n)\in\mathds{N_0}^n$ we set
$\abs{\alpha}~=~\alpha_1~+~\ldots~+~\alpha_n$ and
$\alpha!~=\alpha_1!\ldots\alpha_2!,$
 also $z^\alpha =
z_1^{\alpha_1}\ldots z_n^{\alpha_n}$ and $\partial^\alpha f =
\frac{\partial^{\abs{\alpha}} f}{\partial
\bar{z}_1^{\alpha_1}\ldots\partial \bar{z}_n^{\alpha_n}}.$



Let $\Om$ be a strongly convex domain with a $C^3$-smooth defining
function. We need to consider a family of domains $$\Omega_t =
\left\{z\in\Cn : \rho(z)<t \right\}$$ that are also strongly convex
for each $|t|<\eps,$ where $\eps>0$ is small enough, that is $d^2\rho(z)$ is positive definite when $|\rho(z)|\leq\eps.$ For $z\in \O_{\eps}\setminus\O_{-\eps}$ we denote the nearest point on $\dom$ as $\pr_{\dom}(z).$ Then the mapping 
$$\pr_{\dom} : \O_{\eps}\setminus\O_{-\eps}\to \dom$$
is well defined, $C^2-$smooth on $\O_{\eps}\setminus\O$ and $|z-\pr_{\dom}(z)|=\dist{z}{\dom}.$

For $\xi\in\dom_t$ we define the complex tangent space
$$T_\xi = \left\{ z\in\Cn : \scp{\partial{\rho}(\xi)}{\xi-z} = 0  \right\}.  $$

The space of holomorphic functions we denote as $H(\O).$
Throughout this paper we use notations $\lesssim,\ \asymp.$  We let
$f\lesssim g$ if $f\leq c g$ for some constant $c>0,$ that doesn't
depend on main arguments of functions $f$ and $g$ and usually depend
only on dimension $n$ and domain $\O.$ Also $f\asymp g$ if $c^{-1}
g\leq f\leq c g$ for some $c>1.$

\section{Cauchy-Leray-Fantappi\`{e} formula \label{CLF}}

In the context of theory of several complex variables there is no
unique reproducing formula formula, however we could use the Leray
theorem, that allows us to construct holomorphic reproducing kernels
(\cite{AYu79}, \cite{L59}, \cite{Ra86}). For convex domain $\Om$
this theorem brings us Cauchy-Leray-Fantappi\`{e} formula, and for
$f\in H^1(\Omega)$ and $z\in\O$ we have
\begin{equation} \label{eq:CLF}
 f(z) = K_\O f(z) = \cn \int\limits_{\dom} \frac{f(\xi) \partial\rho(\xi)\wedge(\bar{\partial}\partial\rho(\xi))^{n-1}}{\scp{\partial\rho(\xi)}{\xi-z}^n} = \int\limits_{\dom} f(\xi) K(\xi,z) \omega(\xi),
\end{equation}
where $\omega(\xi) = \cn
\partial\rho(\xi)\wedge(\bar{\partial}\partial\rho(\xi))^{n-1},$ and
$K(\xi,z) =\scp{\partial\rho(\xi)}{\xi-z}^{-n}.$

The $(2n-1)$-form $\omega$ defines on $\dom_t$ Leray-Levy measure
$dS$ that is equivalent to Lebesgue surface measure $d\sigma_t$
(for details see \cite{AYu79}, \cite{LS13}, \cite{LS14}). This
allows us to identify Lebesgue spaces
defined with respect to measures $d\sigma_t$ and $dS$. Also note,
that measure $dV$ defined by the $2n$-form
$d\omega=(\partial\dbar\rho)^{n}$ is equivalent to Lebesgue measure
$d\mu$ in $\Cn.$

By \cite{R12} the integral operator $K_\O$ defines a bounded mapping
on $L^p(\dom)$ to $H^p(\Omega)$ for $1<p<\infty.$

The function $d(w,z) = \abs{\V{w}{z}}$  defines on $\dom$
quasimetric, and if $B(z,\delta) = \{w\in\dom: d(w,z)<\delta\}$ is a
quasiball with respect to $d$ then $\sigma(B(z,\delta))\asymp
\delta^n,$ see for example \cite{R12}. Therefore $\{\dom,d,\sigma\}$
is a space of homogeneous type.

Note also the crucial role in the forthcoming considerations of the
following estimate that is proved in \cite{R13}.
\begin{lemma} \label{lm:QM_est1} Let $\O$ be strongly convex, then
$$d(w,z) \asymp \rho(w) + d(\pr_{\dom}(w),z),\ w\in\CO,\ z\in\dom.$$
\end{lemma}

\subsection{Kor\'{a}nyi regions}
For $\xi\in\dom$ and $\eps>0$ we define the {\it inner Kor\'{a}nyi
region} as
$$D^i(\xi,\eta,\eps) = \{\tau\in\O : \pr_{\dom}(\tau) \in B(\xi,-\eta\rho(\tau)),\ \rho(\tau)>-\eps \}. $$

The strong convexity of $\O$ implies that area-integral inequality
by S.~Krantz and S.Y.~Li~\cite{KL97} for $f\in H^p(\O),\
0<p<\infty,$ could be expressed as

\begin{equation} \label{ineq:Luzin_internal}
  \intl_\dom d\sigma(z) \left( \intl_{D^i(z,\eta,\eps)} \abs{\partial f(\tau)}^2
  \frac{d\mu(\tau)}{(-\rho(\tau))^{n-1}}\right)^{p/2} \leq c(\O,p) \intl_\dom \abs{f}^p d\sigma.
\end{equation}

Consider the decomposition of vector $\tau\in\Cn$ as $\tau = w +
t n(\xi),$ where $w\in T_\xi,\ t\in\C,$ and
$n(\xi)=\frac{\bar\partial\rho(\xi)}{\abs{\bar\partial\rho(\xi)}}$ is a
complex normal vector at $\xi$. We define the {\it external
Kor\'{a}nyi region} as
\begin{multline}\label{df:KoranyExt}
 D^e(\xi,\eta,\eps) = \{\tau\in\CO : \tau = w + t n(\xi),\\ w\in T_\xi,\
t\in\C,\ \abs{w}<\sqrt{\eta\rho(\tau)},\ \abs{\IM(t)}<\eta\rho(\tau),\
\rho(\tau)<\eps \}.
\end{multline}
The main result of this paper is the area-integral inequality similar to~(\ref{ineq:Luzin_internal}) for external
regions~$D^e(\xi,\eta,\eps).$

We point out two rules for integration over regions $D^e(\xi,\eta,\eps).$
First, for every function $F$ we have
$$\intl_{\O_\eps\setminus\O} \abs{F(z)} d\mu(z) \asymp \intl_\dom d\sigma(\xi)
\intl_{D^e(\xi,\eta,\eps)} \abs{F(\tau)}
\frac{d\mu(\tau)}{\rho(\tau)^{n}}.
$$
Second, if $F(w) = \tilde{F}(\rho(w))$ then
$$ \intl_{D^e(\xi,\eta,\eps)} \abs{F(\tau)} d\mu(\tau) \asymp \intl_0^\eps \abs{\tilde{F}(t)} t^{n} dt.  $$
Similar rules are valid for regions $D^i(\xi,\eta,\eps).$

We could clarify the estimate of $d(\tau,w)$ in
lemma~\ref{lm:QM_est1} for $\tau\in D^e(z,\eta,\eps).$

\begin{lemma} \label{lm:QM_est2} Let $\O$ be a strongly convex domain and $\eps,\eta>0,$
then
\begin{equation*} \label{ineq:QM_est2}
 d(\tau,w) \asymp \rho(\tau) + d(z,w),\quad z,w\in\dom,\ \tau\in
 D^e(z,\eta,\eps).
\end{equation*}
\end{lemma}

\begin{proof}
For $\tau\in D^e(z,\eta,\eps)$ we denote $\hat{\tau}=\pr_{\dom}(\tau),$ then $d(\hat{\tau},z)\lesssim
\eta\rho(\tau)$ and by lemma~\ref{lm:QM_est1}
\begin{equation*}
d(\tau,w)\lesssim \rho(\tau)+d(\hat{\tau},w)\lesssim
\rho(\tau)+d(\hat{\tau},z) +d(z,w) \lesssim \rho(\tau) +d(z,w).
\end{equation*}
On the other hand,
\begin{multline*} \rho(\tau) +d(z,w) \lesssim \rho(\tau) +
(d(z,\hat{\tau})+d(\hat{\tau},w))\lesssim
(1+\eta)\rho(\tau)+d(\hat{\tau},w)\\
\lesssim \rho(\tau)+d(\hat{\tau},w) \lesssim d(\tau,w). 
\end{multline*}
\qed
\end{proof}

\section{Area-integral inequality for external Kor\'{a}nyi region
}\label{Area_int} \numberwithin{theorem}{section}

Let $\O\subset\Cn$ be a strongly convex domain and $\eta,\eps>0$. For
function $g\in L^1(\dom)$ and $l\in\N$ we define a function
\begin{equation} \label{eq:area_int}
I_l(g,z) = \left(\ \intl_{D^e(z,\eta,\eps)} \abs{\ \intl_\dom \frac{ g(w)
dS(w)}{\V{\tau}{w}^{n+l}} }^2 d\nu_l(\tau) \right)^{1/2},
\end{equation}
where $dS(w) =\cn
\partial\rho(w)\wedge(\bar{\partial}\partial\rho(w))^{n-1}$ (see (\ref{CLF})) and  $d\nu_l(\tau) = \frac{d\mu(\tau)}{\rho(\tau)^{n-2l-1}}.$

\begin{theorem}\label{thm:area_int}
Let $\O$ be strongly convex domain and $g\in L^p(\dom),\
1<p<\infty,$ Then
\begin{equation} \label{est:area_int}
    \intl_\dom I_l(g,z)^p d\sigma(z)  \lesssim \intl_\dom \abs{g(z)}^p
d\sigma(z).
\end{equation}

\end{theorem}
Note that in the one-variable case the integral~(\ref{eq:area_int})
is a holomorphic function and the result of the theorem follows
from~\cite{D77}.

Recall that (semi)norm in $\BMO=\BMO(\dom)$ is defined by 
$$\norm{f} = \sup\frac{1}{\sigma(B)} \int\limits_{B}\abs{f-f_{B}}d\sigma,$$
where $f_B = \sup\frac{1}{\sigma(B)} \int\limits_{B} f d\sigma $ is the average value of $f$ on the quasiball $B$ and the supremum is taken over all quasiballs $B\subset\dom.$ We prove that operator $I_l$ is also bounded on $\BMO.$

\begin{theorem}\label{thm:area_int_BMO}
Let $\O$ be strongly convex domain then
\begin{equation} \label{est:area_int_BMO}
    \norm{I_l(g)}_{\BMO} \lesssim \norm{g}_{\BMO}.
\end{equation}
\end{theorem}

The main idea of proof of these theorems is that the operator $I_l$ could be considered as a sum of operators with values in some model $L^2$-space on function (see  decomposition~(\ref{eq:Luzin_decomposition}) and formula~(\ref{eq:kernels}) for kernel).
\begin{definition} Assume, that defining function~$\rho$ for strongly convex domain $\O$
has the following form near $0\in\dom$
\begin{equation} \label{eq:rhostandart}
 \rho(z) = 2\RE(z_n) + \suml_{j,k=1}^n A_{jk} z_j \bar{z}_k + O(\abs{z}^3)
\end{equation}
with positive definite form $A_{jk} z_j \bar{z}_k.$ We define a set
\begin{multline}
 D_0(\eta,\eps) = \{ \tau\in\CO : \abs{\tau_1}^2+\ldots+\abs{\tau_{n-1}}^2 < \eta \RE(\tau_n),\\ \abs{\IM(\tau_n)} <\eta\RE(\tau_n),\ \abs{\RE(\tau_n)}<\eps\}.
\end{multline}
\end{definition}

\begin{lemma}\label{lemma:rhostandart}
 Suppose, that $\rho$ has the form (\ref{eq:rhostandart}). There exist constants $c,\eps_0>0$ such that
    \begin{equation*} 
        D^e(0,\eta,\eps)\subset D_0(c\eta,c\eps),\ D_0(\eta,\eps)\subset D^e(0,c\eta,c\eps)\ \text{for}\ 0<\eta,\eps<\eps_0.
    \end{equation*}
\end{lemma}

\begin{proof} For the function
$\rho$ of the form~(\ref{eq:rhostandart}) the Kor\'{a}nyi sector (\ref{df:KoranyExt})
could be expressed as follows
\begin{multline*}
D^e(0,\eta,\eps) =  \{\tau\in \CO: \abs{\tau_1}^2+\ldots+\abs{\tau_{n-1}}^2\leq \eta\rho(\tau),\\ \abs{\IM(\tau_n)}\leq\eta\rho(\tau),\ \rho(\tau)<\eps\}
\end{multline*}
and
\begin{multline*}
\rho(\tau)\leq 2\RE(\tau_n) + c_0\left( \abs{\tau_1}^2+\ldots+\abs{\tau_{n-1}}^2 + \IM(\tau_n)^2 + \RE(\tau_n)^2\right)\\
\leq  (2+ c_0 \RE(\tau_n))\RE(\tau_n) + c_0(1+ \eta\rho(\tau))\eta\rho(\tau), \ \tau\in D^e(0,\eta,\eps) .
\end{multline*}
Thus for $\eta<\eta_0=\frac{1}{8c_0}$ we have $\rho(\tau)\leq c \RE(\tau_n).$

It is easy to see, that $|\tau|\to 0$ when $\rho(\tau)\to 0, \tau\in D^e(0,\eta,\eps).$  Then by convexity of $\O$
$$2\RE(\tau_n) =  \rho(\tau) - \suml_{j,k=1}^n A_{jk} \tau_j \bar{\tau}_k + O(\abs{\tau}^3)\leq \rho(\tau),\ \tau\in D^e(0,\eta,\eps_0) $$
for some $\eps_0\in(0,\eta_0).$

Finally $D^e(0,\eta,\eps)\subset D_0(c\eta,\eps)$ and analogously $D_0(\eta,\eps)\subset D^e(0,\eta,\eps)$ for $0<\eta,\eps<\eps_0.$ \qed

\end{proof}

\begin{theorem} \label{thm:rhostandart}
There exists such covering of the set 
$\overline{\O}_\eps\setminus\O_{-\eps}$ by open sets $\Gamma_j$ such
that for every $\xi\in\Gamma_j$ we can find a holomorphic change of
coordinates $\varphi_j(\xi,\cdot) : \Cn \to \Cn $ such that
\begin{enumerate}
    \item[{\rm{1.}}] The mapping $\varphi_j(\xi,\cdot)$
    transforms function $\rho$ to the type (\ref{eq:rhostandart})
    and could be expressed as follows
    \begin{equation}\label{thm:rhostandart:cond1}
    \varphi_j(\xi,z) = \Phi_j(\xi) (z-\xi) + (z-\xi)^\perp B_j(\xi) (z-\xi) e_n,
    \end{equation}
    where matrices  $\Phi_j(\xi), B_j(\xi)$ are $C^1$-smooth on
    $\Gamma_j,$ and $e_n=(0,\ldots,0,1).$
    \item[{\rm{2.}}] Let $\psi_j(\xi,\cdot)$ be an inverse map of $\varphi_j(\xi,\cdot),$ and let $J_j(\xi,\cdot)$ be a complex Jacobian of $\psi_j$.
    Then
    \begin{align} \label{thm:rhostandart:cond2}
     \sup\limits_{\tau\in\O_\eps\setminus\overline{\O}_\eps}
     \abs{J_j(\xi,\cdot) - J_j(\xi',\cdot)} &\lesssim \abs{\xi-\xi'},\\
     \sup\limits_{\tau\in\O_\eps\setminus\overline{\O}_\eps}
     \abs{\psi_j(\xi,\cdot) - \psi_j(\xi',\cdot)} &\lesssim \abs{\xi-\xi'}.\
    \end{align}
    Note that real Jacobian is then equal to $\abs{J_j(\xi,\cdot)}^2 = J_j(\xi,\cdot)\overline{J_j(\xi,\cdot)}.$
    \item[{\rm{3.}}] There exist constants $c,\eps_0>0$
    such that for $0<\eta,\eps<\eps_0$
    \begin{equation} \label{thm:rhostandart:cond3}
        \varphi_j(\xi,D^e(\xi,\eta,\eps))\subseteq D_0(c\eta,c\eps),\quad \psi_j(\xi,D_0(\eta,\eps)\subseteq
        D^e(\xi,c\eta,c\eps).
    \end{equation}
\end{enumerate}
\end{theorem}

\begin{proof}
Let $\xi\in\dom,$ by linear change of coordinates $z' =
(z-\xi)\Phi(\xi)$ we could obtain the following form for function
$\rho$ 
\begin{multline*}
\rho(z) = \rho(\xi+ \Phi^{-1}(\xi)z')\\ = 2\RE(z'_n) + \suml_{j,k=1}^n A^1_{jk}(\xi) z'_j\bar{z}'_k  + \RE\suml_{j,k=1}^n A_{jk}^2(\xi) z'_j z'_k  +O(\abs{z'}^3).
\end{multline*}
Setting $z''_n =z'_n +  A_{jk}^2 z'_j z'_k$ and $z''_j = z'_j,\
1\leq j\leq n-1,$ we have (see \cite{Ra86})
$$\rho(z'')= 2\RE(z'_n) + \suml_{j,k=1}^n A^1_{jk}(\xi) z''_j\bar{z}''_k   +O(\abs{z''}^3).  $$
Denote $B(\xi) = \Phi(\xi)^\perp A^2(\xi) \Phi(\xi),$ then
$$\varphi(\xi,z) = \Phi(\xi) (z-\xi) + (z-\xi)^\perp B(\xi)
(z-\xi) e_n.$$

We choose $\Gamma_j$ such that the matrix $\Phi(\xi)$ could be
defined on $\Gamma_j$ smoothly, this choice we denote as $\Phi_j,$
and the change corresponding to this matrix as $\varphi_j$
$$\varphi_j(\xi,z) = \Phi_j(\xi) (z-\xi) + (z-\xi)^\perp B_j(\xi)
(z-\xi)e_n.$$ Thus mappings $\varphi_j$ satisfy the first condition.
Easily, the second condition also holds.

The last condition (\ref{thm:rhostandart:cond3}) follows immediately from lemma~\ref{lemma:rhostandart}. This ends the
proof of the theorem. \qed
\end{proof}

Further we will assume, that the covering
$\overline{\O}_\eps\setminus\O_{-\eps} \subset
\bigcup\limits_{j=1}^N\Gamma_j$ and maps $\varphi_j, \psi_j$ are
chosen by the theorem~\ref{thm:rhostandart}. For covering
$\{\Gamma_j\}$ we consider a smooth decomposition of identity on
$\dom:$
$$\chi_j \in C^{\infty}(\Gamma_j),\ 0\leq\chi_j\leq 1,\ \supp{\chi_j}\subset\Gamma_j,\ \suml_{j=1}^N \chi_j(z) = 1,\ z\in\dom.$$

Fix parameters $0<\eps,\eta<\eps_0,$ denote $D_0 =D_0(\eta,\eps).$ Then by
(\ref{thm:rhostandart:cond3}) 
$$D^e(z)=\varphi_j(z,D^e(z,\eta/c,\eps/c))\subset D_0$$ and
\begin{multline} \label{eq:Luzin_decomposition}
I_l (g,z)^2\\ = \suml_{j=1}^N \chi_j(z)
\intl_{D^e(z)} \abs{\ \intl_\dom \frac{
g(w)J_j(z,\tau) dS(w)}{\V{\psi_j(z,\tau)}{w}^{n+l}} }^2
\frac{d\mu(\tau)}{\RE(\tau_n)^{n-2l+1}} \\
 \lesssim \suml_{j=1}^N  \intl_{D_0} \abs{\ \intl_\dom \frac{ g(w)\chi_j^{1/2}(z) J_j(z,\tau)
dS(w)}{\V{\psi_j(z,\tau)}{w}^{n+l}} }^2
\frac{d\mu(\tau)}{\RE(\tau_n)^{n-2l+1}}.
\end{multline}

\noindent We will consider the function
\begin{equation} \label{eq:kernels}
 K_j(z,w) (\tau) = \frac{\chi_j^{1/2}(z)J_j(z,\tau)}{\V{\psi_j(z,\tau)}{w}^{n+1}}
\end{equation}
as a map $\dom\times\dom\to \mathscr{L}(\C,L^2(D_0,d\nu_l)),$ such
that its values are operator of multiplication from $\C$ to
$L^2(D_0,d\nu_l),$ where
$d\nu_l(\tau)=\frac{d\mu(\tau)}{\IM(\tau_n)^{n-2l+1}}$ is a measure
on the region $D_0.$ Throughout the proof of the
theorem~\ref{thm:area_int} $j,l$ will be fixed integers and the norm
of function $F$ in the space $L^2(D_0,d\nu_l)$ will be denoted as~$\norm{F}.$

We will show that integral operator $T_j$ defined by kernel $K_j$ is
bounded on $L^p.$ To prove this we apply $T1$-theorem for
transformations with operator-valued kernels formulated by
Hyt\"{o}nen and Weis in \cite{HW05}, taking in account that in our
case concerned spaces are Hilbert. Some details of the proof are
similar to the proof of the boundedness of operator
Cauchy-Leray-Fantappi\`{e} $K_\O$ for lineally convex domains
introduced in \cite{R12}. Below we formulate the $T1$-theorem,
adapted to our context.

\begin{definition}
We say that the function $f\in C^\infty_0(\dom)$ is a normalized
bump-function, associated with the quasiball $B(w_0,r)$ if
$\supp{f}\subset B(w_0,r),$ $\abs{f}\leq 1,$ and
$$\abs{f(\xi)-f(z)}\leq \frac{d(\xi,z)^{\gamma}}{r^{\gamma}},\ \xi,z\in\dom.$$
The set of bump-functions associated with $B(w_0,r)$ is denoted as
$A(\gamma,w_0,r).$
\end{definition}

\begin{theorem} \label{thm:T1}
Let $K:\dom\times\dom\to \mathscr{L} (\C, L^2(D_0,d\nu_l))$ verify
the estimates
\begin{align}
    &\norm{K(z,w)} \lesssim \frac{1}{d(z,w)^n}; \label{KZ1}\\
    &\norm{K(z,w)-K(\xi,w)} \lesssim \frac{d(z,\xi)^\gamma}{d(z,w)^{n+\gamma}},\quad d(z,w)> C d(z,\xi); \label{KZ2}\\
    &\norm{K(z,w)-K(z,w')} \lesssim \frac{d(w,w')^\gamma}{d(z,w)^{n+\gamma}},\quad d(z,w)> C d(w,w')\label{KZ3}
\end{align}
for $\xi,z,w\in\dom$ and some constant $C>0.$

Assume that operator $T:
\mathscr{S}(\dom)\to\mathscr{S}'(\dom,\mathscr{L} (\C,
L^2(D_0,d\nu_l)))$ with kernel $K$ verify the following conditions.
\begin{itemize}
 \item $T1,\ T'1\in\BMO(\dom,L^2(D_0,d\nu_l)),$ where $T'$ is
 formally adjoint operator.
 \item Operator $T$ satisfies the weak boundedness property, that is
 for every pair of normalized bump-functions $f,g\in A(\gamma,w_0,r)$ we have $$\norm{\scp{g}{Tf}}\leq C
 r^{-n}.$$
\end{itemize}
Then $T\in\mathscr{L} (L^p(\dom), L^p(\dom,L^2(D_0,d\nu_l))$ for
every $p\in (1,\infty).$
\end{theorem}
\medskip

In the following four lemmas we will prove that kernels $K_j$ and
corresponding operators $T_j$ satisfy the conditions of the
$T1$-theorem. In particular, in lemmas~\ref{lm:T1_pre},\ref{lm:T1_3} we prove that 
$T1,T'1\in L^\infty (\dom,L^2(D_0,d\nu_l))\subset\BMO(\dom,L^2(D_0,d\nu_l)).$

\begin{lemma} \label{lm:T1_1} The kernel $K_j$ verify estimates (\ref{KZ1}-\ref{KZ3}).

\end{lemma}
\begin{proof} By lemma \ref{lm:QM_est2} we have 
$\abs{\V{\tau}{w}}\asymp\rho(\tau) + \abs{\V{z}{w}},$ $z,w\in\dom,\
\tau\in D^e(z,c\eta,c\eps). $
Thus
\begin{multline*}
\norm{ K_j(z,w)}^2 = \intl_{D_0} \abs{K_j(z,w)(\tau)}^2 d\nu_l(\tau)
\lesssim\intl_{D^e(z,c\eta,c\eps)}
\frac{d\nu_l(\tau)}{ \abs{ \V{\tau}{w} }^{2n+2l} } \\
 \lesssim\intl_{D^e(z,c\eta,c\eps)}\frac{1}{( \rho(\tau) + \abs{\V{z}{w}} )^{2n+2l}}
\frac{d\mu(\tau)}{\rho(\tau)^{n-2l+1}}\\
\lesssim \intl_0^\infty\frac{t^{2l-1} dt}{(t+\abs{\V{z}{w}})^{2n+2l}} \lesssim
\frac{1}{\abs{\V{z}{w}}^{2n}}\lesssim \frac{1}{d(z,w)^{2n}}.
 \end{multline*}

Similarly,
\begin{multline*}\norm{K_j(z,w) - K_j(z,w')}^2\\
\lesssim\intl_{D^e(z,c\eta,c\eps)} \abs{\frac{1}{\V{\tau}{w}^{n+l}} -
\frac{1}{\V{\tau}{w'}^{n+l}}}^2 d\nu_l(\tau).
\end{multline*}

%

\noindent Denote $\hat{\tau} =\pr_{\dom}(\tau),$ then
\begin{multline*} \abs{\V{\tau}{w}} \lesssim \rho(\tau) +
\abs{\V{\hat{\tau}}{w}}\\ \lesssim \rho(\tau) +  \abs{\V{z}{w}} +
\abs{\V{\hat{\tau}}{z}} \lesssim \rho(\tau) + \abs{\V{z}{w}},
\end{multline*}
which combined with lemma \ref{lm:QM_est2} and estimate $$d(w,w')=\abs{\V{w}{w'}}<
C \abs{\V{z}{w}}=C d(z,w)$$ implies
\begin{multline*}
\abs{\V{\tau}{w}} \asymp \rho(\tau) +  \abs{\V{z}{w}} \asymp \rho(\tau) +  \abs{ \V{z}{w'} }\\ \asymp \abs{\V{\tau}{w'}}.  
\end{multline*}
Next, we have 
\begin{multline*}
 \abs{ \V{\tau}{w'}-\V{\tau}{w} } =  \abs{ \scp{\partial\rho(\tau)}{\hat{\tau}-w}-\scp{\partial\rho(\tau)}{\hat{\tau}-w'} }\\
  \leq \abs{ \scp{\partial\rho(\tau)-\partial\rho(\hat{\tau})}{w-w'} } +
  \abs{ \scp{\partial\rho(\hat{\tau})}{\hat{\tau}-w}-\scp{\partial\rho(\hat{\tau})}{\hat{\tau}-w'} }\\
  \lesssim \rho(\tau) \abs{\V{w}{w'}}^{1/2} +
  \abs{\V{\hat{\tau}}{w}}^{1/2}\abs{\V{w}{w'}}^{1/2}\\ \lesssim   \abs{\V{\tau}{w}}^{1/2} \abs{\V{w}{w'}}^{1/2}
\end{multline*}
Hence,
\begin{multline*}\norm{K_j(z,w) - K_j(z,w')}^2\lesssim
\intl_{D^e(z,c\eta,c\eps)} \frac{ \abs{\V{w}{w'}} }{ \abs{\V{\tau}{w}}^{2n+2l+1} } d\nu_l(\tau)\\
\lesssim  \intl_0^\infty \frac{\abs{\V{w}{w'}} t^{2l-1}
dt}{(t+\abs{\V{z}{w}})^{2n+2l+1}} \lesssim \frac{ \abs{\V{w}{w'}} }{
\abs{ \V{z}{w} }^{2n+1} }=\frac{d(w,w')}{d(z,w)^{2n+1}}.
\end{multline*}

The last inequality (\ref{KZ3}) is a bit harder to prove. 

Let $z,\xi,w\in\dom,\ Cd(z,\xi)<d(z,w),$ and estimate the value
$$A = \abs{\V{\psi_j(z,\tau)}{w} - \V{\psi_j(\xi,\tau)}{w}  }. $$
Denote $\tau_z = \psi_j(z,\tau),\ \tau_\xi =\psi_j(\xi,\tau),$ then by (\ref{thm:rhostandart:cond1})
\begin{multline*}
\tau = \Phi(z)(\tau_z-z)+i (\tau_z-z)^{T}B(z)(\tau_z-z)e_n\\ = \Phi(\xi)(\tau_\xi-\xi)+i
(\tau_\xi-\xi)^{T}B(\xi)(\tau_\xi-\xi)e_n,
\end{multline*} whence denoting $\Psi(z) = \Phi(z)^{-1}$ and introducing $L(z,\xi,\tau)$ we obtain
\begin{align*}
\tau_z &= z+ \Psi(z)\tau -  (\tau_z-z)^{T}B(z)(\tau_z-z)\Psi(z)e_n ,\\
\tau_\xi &= \xi+ \Psi(\xi)\tau - (\tau_\xi-\xi)^{T}B(\xi)(\tau_\xi-\xi)\Psi(\xi)e_n ,\\
\tau_z-\tau_\xi &= z-\xi + (\Psi(z)-\Psi(\xi))\tau + L(z,\xi,\tau)e_n .
\end{align*}
Note, that norms of matrices $\norm{\Psi(\xi)}$ are bounded, thus
\begin{multline*}
 \abs{L(z,\xi,\tau)} \leq \abs{(\tau_z-z)^{T}B(z)(\tau_z-z)(\Psi(z)-\Psi(\xi))}\\ + \abs{(\tau_z-z)^{T}B(z)(\tau_z-z) -
 (\tau_\xi-\xi)^{T} B(\xi)(\tau_\xi-\xi)} \norm{\Psi(\xi)}\\
 \lesssim \abs{z-\xi}\abs{\tau_z-z}^2 + \abs{(\tau_z-z-\tau_\xi+\xi)^T
 B(z)(\tau_z-z)}\\
 + \abs{(\tau_\xi-\xi)^T B(z)(\tau_z-z) - (\tau_\xi-\xi)^T B(\xi)(\tau_\xi-\xi)}\\
  \lesssim \abs{z-\xi}\abs{\tau_z-z}^2 + \abs{z-\xi}\abs{\tau} + \abs{ ((\Psi(z)-\Psi(\xi))\tau+L(z,\xi,\tau)e_n)^T
  B(z)(\tau_z-z)}\\ + \abs{(\tau_\xi-\xi)^T(B(z)-B(\xi))(\tau_z-z)} + \abs{(\tau_\xi-\xi)^T B(\xi)
  (\tau_z-z-\tau_\xi-\xi)}\\
  \lesssim \abs{z-\xi}\abs{\tau_z-z}^2 + \abs{z-\xi}\abs{\tau} +\abs{\tau}
  \abs{L(z,\xi,\tau)}
  +\abs{z-\xi}\abs{\tau}^2 + \abs{\tau} L(z,\xi,\tau).
\end{multline*}

Choosing $\eps>0$ small enough we get 
$\abs{\tau}\leq\eta\abs{\IM(\tau_n)}+(1+\eta)\abs{\IM(\tau_n)}\leq3{\eps}$
and
$\abs{L(z,\xi,\tau)} \lesssim d(z,\xi)^{1/2}\abs{\tau},$ for $\tau\in D_0=D_0(\eta,\eps).$ Hence,
\begin{multline*}
 A\leq
 \abs{\scp{\partial\rho(\tau_z)-\partial\rho(\tau_\xi)}{\tau_z-w}} +
 \abs{\scp{\partial\rho(\tau_\xi)}{\tau_z-w}}\\
 \lesssim \abs{\tau_z-\tau_\xi}(\rho(\tau_z)+d(z,w)^{1/2}) +
 \abs{\scp{\partial\rho(\tau_z)-\partial\rho(\tau_\xi)}{z-\xi}} +
 \abs{\V{z}{\xi}}\\
 + \abs{\scp{\partial\rho(\tau_\xi)}{(\Psi(z)-\Psi(\xi))\tau}} + \abs{\scp{\partial\rho(\tau_\xi)}{L(z,\xi,\tau)}} \lesssim d(z,\xi)^{1/2}d(\tau_z,w) +\\
 \abs{\tau_z-\xi}\abs{z-\xi}+d(z,\xi)+\abs{z-\xi}\abs{\tau}+\abs{L(z,\xi,\tau)} \lesssim d(z,\xi) + d(z,\xi)^{1/2}d(z,w)^{1/2}\\
 \lesssim  d(z,\xi)^{1/2}d(z,w)^{1/2}
\end{multline*}
Combining this estimate with inequality
$\abs{\V{\tau_z}{w}}\asymp\abs{\V{\tau_\xi}{w}}$ we obtain
\begin{multline*}
\norm{K_j(z,w) - K_j(\xi,w)}^2 \lesssim
 \intl_{D^e(z,c\eta,c\eps)}
\frac{\abs{\chi_j(z)^{1/2}-\chi_j(\xi)^{1/2}}^2}{\abs{\V{\tau}{w}}^{2n+2l}} \frac{d\mu(\tau)}{\rho(\tau)^{n-2l+1}}\\
+ \chi_j(\xi) \intl_{D_0}
\frac{ \abs{\V{z}{\xi}}\abs{\V{z}{w}} }{ \abs{\V{\tau_z}{w}}^{2n+4} }\frac{d\mu(\tau)}{\RE(\tau_n)^{n-2l+1}}\\
\lesssim \frac{ \abs{\V{z}{\xi}} }{ \abs{\V{z}{w}}^{2n} } + \frac{
\abs{\V{z}{\xi}} }{ \abs{\V{z}{w}}^{2n+1} } \lesssim \frac{
\abs{\V{z}{\xi}} }{ \abs{\V{z}{w}}^{2n+1} }\\ 
\lesssim \frac{d(z,\xi)}{d(z,w)^{2n+1}}.
\end{multline*}\qed
\end{proof}

\begin{lemma}\label{lm:T1_pre} Let $\tau_z = \psi_j(z,\tau),$ then 
$$\norm{ \intl_\dom \frac{dS(w)}{\V{\tau_z}{w}^{n+l}} }\lesssim 1$$
and, consequently, $\norm{T_j(1)}\lesssim 1.$ 
\end{lemma}

\begin{proof}
The function
$\V{\tau_z}{w}$ is holomorphic in $\O$ with respect to $w,$ then 
\begin{equation}\label{eq:T1}
T_j(1)(\tau) =  \intl_\dom \frac{\chi_j(z)^{1/2} J_j(z,\tau)dS(w)}{\V{\tau_z}{w}^{n+l}} =  \intl_\O \frac{\chi_j(z)^{1/2} J_j(z,\tau)dV(w)}{\V{\tau_z}{w}^{n+l}}. 
\end{equation}

Analogously to lemma \ref{lm:QM_est2} we have $\abs{\V{\tau_z}{w}}
\asymp \IM(\tau_n) + |\rho(w)| + \abs{\V{z}{\hat{w}}},$ where $\hat{w}=\pr_{\dom}(w).$

Hence,
\begin{multline}\label{L2est1}
 \abs{T_j(1)(\tau)}\lesssim \intl_{\O}
\frac{d\mu(z)}{ \abs{\V{\tau_z}{w}}^{n+l} }\\
\lesssim\intl_0^{T}
dt\intl_{\dom_t}\frac{d\sigma_t}{(t+\IM(\tau_n)+ \abs{\V{z}{\hat{w}}} )^{n+l}}\\
\lesssim \intl_0^{T} dt \intl_0^\infty
\frac{v^{n-1}dv}{(t+\RE(\tau_n)+v)^{n+l}}\lesssim \intl_0^{T}
\frac{dt}{(t+\RE(\tau_n))^l}\\ \lesssim (\RE(\tau_n))^{1-l}
\ln{\left(1+\frac{1}{\RE(\tau_n)}\right)},
\end{multline}
and
\begin{multline} \label{L2est2}
\intl_{D_0} \abs{T_j(1)(\tau)}^2d\nu_l(\tau) \lesssim \intl_{D_0} (\RE(\tau_n))^{2-2l}\ln^2\left(1+\frac{1}{\RE(\tau_n)}\right)d\nu_l(\tau)\\ \lesssim\intl_0^{\eps} \ln^2{\left(1+\frac{1}{s}\right)} s ds\lesssim 1.
\end{multline}
This finishes the proof of the lemma.

\end{proof}

\begin{lemma} \label{lm:T1_2} $\norm{T_j'(1)} \lesssim 1.$
\end{lemma}

\begin{proof} Consider
\begin{multline*}
 \overline{T'_j(1)}(w)(\tau)= \intl_\dom \frac{\chi_j(z)^{1/2}
J_j(z,\tau) dS(z)}{\V{\tau_z}{w}^{n+l}}\\ = \intl_\dom
\frac{\chi_j(z)^{1/2} J_j(z,\tau)
(dS(z)-dS(\tau_z))}{\V{\tau_z}{w}^{n+l}} + \intl_\dom
\frac{\chi_j(z)^{1/2} J_j(z,\tau) dS(\tau_z)}{\V{\tau_z}{w}^{n+l}} =
L_1+L_2.
\end{multline*}
Note that $ \abs{z-\tau_z}\lesssim \RE(\tau_n),$ therefore
$\abs{dS(z)-dS(\psi(z,\tau))} \lesssim \RE(\tau_n)d\sigma(z)$ and
\begin{multline*}
\abs{L_1} \lesssim \intl_\dom \frac{\RE(\tau_n)d\sigma(z)}{
\abs{\V{\tau_z}{w}}^{n+l} } \lesssim
\frac{\RE(\tau_n)d\sigma(z)}{(\RE(\tau_n)+ \abs{\V{z}{w}})^{n+l} }\\
\lesssim \intl_0^\infty \frac{\RE(\tau_n)
v^{n-1}dv}{(\RE(\tau_n)+v)^{n+l}}\lesssim\frac{1}{\RE(\tau_n)^{l-1}}.
\end{multline*}
Thus we get
\begin{equation} \label{T'1_I1}
\intl_{D_0} \abs{L_1}^2 d\nu_l(\tau) \lesssim \intl_{D_0}
\frac{1}{\RE(\tau_n)^{2l-2}}\frac{d\mu(\tau)}{\RE(\tau_n)^{n-2l+1}}
\lesssim\intl_0^{\eps} \frac{t^ndt}{t^{n-1}} \lesssim 1
\end{equation}
To estimate $L_2$ we recall that $d_\xi
\frac{dS(\xi)}{\V{\xi}{z}^{n}} =0,\ z\in\dom,\ \xi\in\CO,$  and
consequently
\begin{multline*}
d_\xi \frac{dS(\xi)}{\V{\xi}{z}^{n+l}} =
\frac{(\dbar\partial\rho(\xi))^n}{\V{\xi}{z}^{n+l}}\\
 - (n+l)\frac{(\dbar_\xi\left(\V{\xi}{z}\right)\wedge\dbar\partial\rho(\xi))^{n-1}}{\V{\xi}{z}^{n+l}}
= -\frac{l}{n} \frac{dV(\xi)}{\V{\xi}{z}^{n+l}}.
\end{multline*}
By Stokes' theorem we obtain
\begin{multline*}
L_2= \intl_\dom \frac{\chi_j(z)^{1/2} J_j(z,\tau)
dS(\tau_z)}{\V{\tau_z}{w}^{n+l}}\\ = \intl_{\O_{\eps_1}\setminus\O}
\frac{\dbar_z\left(\chi_j(z)^{1/2} J_j(z,\tau)\right)\wedge
dS(\tau_z)}{\V{\tau_z}{w}^{n+l}} -\frac{l}{n}
\intl_{\O_{\eps_1}\setminus\O} \frac{\chi_j(z)^{1/2} J_j(z,\tau)
dV(\tau_z)}{\V{\tau_z}{z}^{n+l}}
\end{multline*}
Again similarly to lemma \ref{lm:QM_est2} we have $\abs{\V{\tau_z}{w}}
\asymp \IM(\tau_n) + \rho(z) + \abs{\V{\hat{z}}{w}},$ where $\hat{z}=\pr_{\dom}(z),$
and the estimate $\norm{L_2}\lesssim 1$ is proven analogously to lemma \ref{lm:T1_pre}. Combining this with the estimate (\ref{T'1_I1}) we get  $\norm{T_j(1)} \lesssim 1. $ \qed

\end{proof}

\begin{lemma} \label{lm:T1_3} Operator $T_j$ is weakly bounded.
\end{lemma}
\begin{proof}
 Let $f,g\in A(\frac{1}{2},w_0,r),$ denote again
$\tau_z=\psi_j(z,\tau),$ then
$$\norm{\scp{g}{T_jf}}^2 \lesssim \intl_{D_0} d\nu_l(\tau) \left( \intl_{B(w_0,r)} \abs{g(z)} dS(z) \abs{\intl_{B(w_0,r)}\frac{f(w) dS(w)}{\V{\tau_z}{w}^{n+l}}
} \right)^2.$$

Denote $t:=\inf\limits_{w\in\dom} \abs{\V{\tau_z}{w}} $ and
introduce the set
$$W(z,\tau,r) :=\left\{ w\in\dom: \abs{\V{\tau_z}{w}}<t+r\right\}.$$
Note that $\supp{f}\subset B(w_0,r)\subset W(z,\tau,cr) \subset B(z,c^2r)$ for some
$c>0,$ therefore,
\begin{multline*}
\abs{\intl_{B(w_0,r)}\frac{f(w) dS(w)}{\V{\tau_z}{w}^{n+l}} }\\
=\abs{\intl_{W(z,\tau,cr)}\frac{f(w)
dS(w)}{\V{\tau_z}{w}^{n+l}} } \lesssim
\intl_{W(z,\tau,cr)}\frac{\abs{f(z)-f(w)} dS(w)}{
\abs{\V{\tau_z}{w}}^{n+l} }\\ + \abs{f(z)} \left( \abs{\intl_{\dom\setminus
W(z,\tau,cr)}\frac{dS(w)}{\V{\tau_z}{w}^{n+l}} } +\abs{\intl_{\dom      }\frac{dS(w)}{\V{\tau_z}{w}^{n+l}} }\right)\\ = L_1(z,\tau) +
\abs{f(z)} \left(L_2(z,\tau) + L_3(z,\tau)\right).
\end{multline*}
It follows from the estimate $|f(z)-f(w)|\leq \sqrt{v(w,z)/r}$ that
\begin{multline*}
L_1(z,\tau)\lesssim \frac{1}{\sqrt{r}}\intl_{B(z,c^2r)}
\frac{v(w,z)^{1/2}}{(\RE(\tau_n) + v(w,z))^{n+l}}\lesssim
\frac{1}{\sqrt{r}} \intl_0^{c^2r} \frac{t^{n-1/2}dt}{(\RE(\tau_n) + t)^{n+l}}\\
\lesssim \frac{1}{\sqrt{r}} \intl_0^{c^2r}\frac{dt}{(\RE(\tau_n) +
t)^{l+1/2}} \lesssim
\frac{1}{\sqrt{r}}\left(\frac{1}{\RE(\tau_n)^{l-1/2}}-
\frac{1}{(\RE(\tau_n)+r)^{l-1/2}}\right)\\ = \frac{1}{\sqrt{r}}
\frac{(\RE(\tau_n)+r)^{l-1/2} - r^{l-1/2}}{\RE(\tau_n)^{l-1/2}
(\RE(\tau_n)+r)^{l-1/2} } \lesssim \frac{1}{\sqrt{r}}
\frac{(\RE(\tau_n)+r)^{2l-1} - r^{2l-1}}{
\IM(\tau_n)^{l-1/2}(\RE(\tau_n)+r)^{2l-1}}\\ \lesssim
\frac{1}{\sqrt{r}} \frac{r \RE(\tau_n)^{2l-2} +
r^{2l-1}}{\RE(\tau_n)^{l-1/2}(\RE(\tau_n)+r)^{2l-1}}.
\end{multline*}
Estimating the $L^2(D_0,d\nu_l)-$norm of the function $L_1(z,\tau),$
we obtain
\begin{multline} \label{est:I1}
\intl_{D_0(\tau)} L_1(z,\tau)^2 d\nu_l(\tau)\\ \lesssim
\intl_{D_0(\tau)} \left(  \frac{r
\RE(\tau_n)^{2l-3}}{(\RE(\tau_n)+r)^{4l-2}} +
\frac{r^{4l-3}}{\RE(\tau_n)^{2l-1}(\RE(\tau_n)+r)^{4l-2}} \right)
\frac{d\mu(\tau)}{\RE(\tau_n)^{n-2l+1}}\\
\lesssim r \intl_0^\infty \frac{s^{4l-4}}{(s+r)^{4l-2}} ds +
r^{4l-3} \intl_0^\infty \frac{ds}{(s+r)^{4l-2}}\lesssim 1
\end{multline}
To estimate the second summand $L_2$ we apply the Stokes theorem to the
domain
$$W_0=\left\{ w\in\O:
\abs{\V{\tau_z}{w}} > t+c r\right\}$$ and to the form $\frac{dS(w)}{\V{\tau_z}{w}^{n+l}}$
\begin{multline*}
\intl_{\dom\setminus W(z,\tau,cr)}\frac{dS(w)}{\V{\tau_z}{w}^{n+l}}
=\intl_{W_0}\frac{dV(w)}{\V{\tau_z}{w}^{n+l}}\\ - \intl_{\substack{w\in\O\\ \abs{v(\tau_z,w)} = t+cr}
}\frac{dS(w)}{\V{\tau_z}{w}^{n+l}}\\ =L_4 -\frac{1}{(t+cr)^{2n+2l}}
\intl_{\substack{w\in\O\\ \abs{v(\tau_z,w)} =
t+cr}}\overline{\V{\tau_z}{w}}^{n+l} dS(w).
\end{multline*}
By the proof of lemma \ref{lm:T1_pre} $$\norm{L_4}\leq \intl_{W_0}\frac{dV(w)}{\abs{\V{\tau_z}{w}}^{n+l}}\leq \intl_{\O}\frac{dV(w)}{\abs{\V{\tau_z}{w}}^{n+l}} \lesssim 1.$$ Applying Stokes' theorem again, now to the domain $$\left\{ w\in\O:
\abs{\V{\tau_z}{w}}<t+cr\right\},$$ we obtain
\begin{multline*}
L_5:=\intl_{\substack{w\in\O\\
\abs{v(\tau_z,w)}=t+cr}}\overline{\V{\tau_z}{w}}^{n+l} dS(w)\\ =
-\intl_{\substack{w\in\dom\\ \abs{v(\tau_z,w)}<t+cr}}\overline{\V{\tau_z}{w}}^{n+l} dS(w)\\
+ \intl_{\substack{w\in\O\\
\abs{v(\tau_z,w)}<t+cr}}
\dbar_w\left(\overline{\V{\tau_z}{w}}^{n+l}\right)\wedge dS(w)\\ +
\intl_{\substack{w\in\O\\
\abs{v(\tau_z,w)}<t+cr}} \overline{\V{\tau_z}{w}}^{n+l} dV(w).
\end{multline*}
Since $ \abs{\
\dbar_w\left(\overline{\V{\tau_z}{w}}^{n+l}\right)\wedge dS(w)}
\lesssim \abs{\V{\tau_z}{w}}^{n+l-1}$ we get
$$
\abs{L_5} \lesssim \intl_t^{t+cr} (s^{n+l}s^{n-1} + s^{n+l}s^{n} +
s^{n+l-1}s^n) ds\lesssim \intl_t^{t+cr} s^{2n+l-1} ds \lesssim r
(t+r)^{2n+l-1}.
$$
Note that $t\asymp \rho(\tau_z)\asymp \IM(\tau_n)$ and consequently
\begin{multline} \label{est:I2}
\intl_{D_0} L_5(z,\tau)^2 d\nu_l(\tau) \lesssim
\intl_{D_0}  \left(\frac{r
(\RE(\tau_n)+r)^{2n+l-1}}{(\RE(\tau_n)+r)^{2n+2l}}\right)^2
d\nu_l(\tau)\\ \lesssim \intl_0^\infty \frac{r^2}{(t+r)^{2l+2}}
\frac{t^n dt}{t^{n-2l+1}} = r^2\intl_0^\infty
\frac{t^{2l-1}}{(t+r)^{2l+2}} \frac{t^n dt}{t^{n-2l+1}} \lesssim r^2
\intl_0^\infty \frac{dt}{(r+t)^3} \lesssim 1.
\end{multline}

Summarizing estimates (\ref{est:I1}, \ref{est:I2}), lemma \ref{lm:T1_pre} and condition
$\abs{f(z)}\leq 1,\ z\in\dom,$ we obtain
\begin{multline*}
\norm{\scp{g}{Tf}}^2\\ \leq \intl_{D_0} d\nu_l(\tau) \left(
\intl_{B(w_0,r)} \abs{g(z)} \left( L_1(z,\tau) + |f(z)|(L_2(z,\tau)+L_3(z,\tau))\right) dS(z)\right)^2 \\
\lesssim \norm{g}_{L^1(\dom)}^2
\sup\limits_{z\in\dom}\intl_{D_0}\left( L_1(z,\tau)^2 + L_2(z,\tau)^2+L_3(\tau,z)^2 \right)
d\nu_l(\tau)\\ \lesssim \norm{g}_{L^1(\dom)}^2\lesssim
\abs{B(w_0,r)}^2.
\end{multline*}
The last estimate implies weak boundedness of operator $T$ and
completes the proof of the lemma. \qed
\end{proof}

\begin{proof}[of the theorem \ref{thm:area_int}] Since operators $T_j$ with kernels $K_j$
verify the conditions of $T1$-theorem, we have $T_j\in\mathscr{L}
(L^p(\dom), L^p(\dom,L^2(D_0,d\nu_l))$ and
\begin{multline*}
 \norm{I(g)}_{L^p(\dom)}^p\suml_{j=1}^N \intl_\dom \norm{T_j g(z)}^p dS(z)\\ =
\suml_{j=1}^N \intl_\dom dS(z) \left(\ \intl_{D_0} \abs{\ \intl_\dom
\frac{ g(w)\chi_j^{1/2}(z) J_j(z,\tau)
dS(w)}{\V{\psi_j(z,\tau)}{w}^{n+1}} }^2
\frac{d\mu(\tau)}{\RE(\tau_n)^{n-1}}\right)^p\\ \lesssim
\norm{g}^p_{L^p(\dom)}.
\end{multline*}
Thus by decomposition (\ref{eq:Luzin_decomposition}) $\intl_\dom
I_l(g,z)^p\ d\sigma(z) \lesssim \intl_\dom \abs{g(z)}^p\ d\sigma(z),$
which proves the theorem. \qed
\end{proof}

\section{Boundedness of area-integral $I_l(g,z)$ on $\BMO(\dom).$}

In this section we prove theorem~\ref{thm:area_int_BMO}. Analogously to theorem~\ref{thm:area_int} in is enough to prove that $T_j \in{\mathcal{L}}\left(\BMO(\dom),\BMO(\dom,L^2(D_0,d\nu_l))\right).$

\begin{lemma}\label{lm:BMO} Functions $T_j(1),T_j'(1)$ satisfy H\"{o}lder condition with exponent $1/2.$
\end{lemma}
\proof{ Similarly to lemma~\ref{lm:T1_1} we can prove, that 
$$\norm{ K_j(\xi,w)-K_j(z,w) }\lesssim \frac{ d(\xi,z)^{1/2} }{ d(z,w)^{n+1/2} },\ x,z\in\dom, w\in\O$$
thus by Jensen's inequality and by expression (\ref{eq:T1}) for $T_j(1)$ 
\begin{multline*}
\norm{ T_j(1)(z)-T_j(1)(\xi) }= \norm{ \int_{\O}\left(K_j(\xi,w)-K_j(z,w)\right)dV(w) }\\
\lesssim
  \int_{\O}\norm{K_j(\xi,w)-K_j(z,w)}dV(w) \lesssim \int_{\O}\frac{ d(\xi,z)^{1/2} }{ d(z,w)^{n+1/2} }\lesssim d(\xi,z)^{1/2}.
\end{multline*}
The estimate for $T_j'$ is obtained analogously.
}

\begin{lemma}
 $T_j,T'_j \in{\mathcal{L}}\left(\BMO(\dom),\BMO(\dom,L^2(D_0,d\nu_l))\right).$
\end{lemma}
\proof{
Let $b\in\BMO(\dom)$ and fix a quasiball $B_\eps=B(z_0,\eps)\subset\dom.$ We decompose $b$ as a sum 
$b=b_1+b_2+b_3,$ where $b_1 = b_{B_{C\eps}},\ b_2=
(b-b_1)\chi_{B_C\eps(z_0)}, $ $\chi_{B_{C\eps}(z_0)}$ is a characteristical function of a quasiball $B_{C\eps}(z_0),$ and $C>0$ is large enough. Here we will Use notation $|B_\eps|=dS(B_\eps)$ (recall that  measures $dS$ and $d\sigma$ are equivalent).

By \cite{FS72} we have $\abs{b_1}\leq \norm{b}_{BMO} \log\frac{1}{\eps},$  hence by lemma~\ref{lm:BMO}
\begin{multline*}
   \frac{1}{|B_\eps|} \int\limits_{B_\eps} \norm{T_j b_1 (z) -
   (T_j b_1)_{B_\eps}}dS(z)\lesssim
   \frac{|b_1|}{|B_\eps|^2}\int\limits_{B_\eps}\int\limits_{B_\eps}\norm{T_j 1(z)-T_j1(\xi)}dS(\xi)dS(z)\\
   \lesssim |b_1| \eps^{1/2} \leq \eps^{1/2}\log\frac{1}{\eps}
   \norm{b}_{BMO} \lesssim \norm{b}_{BMO}.   
\end{multline*}

To estimate $T_j b_2$ we use the boundedness of 
\begin{multline*}
\left( \frac{1}{|B_\eps|}\int\limits_{B_\eps} \norm{T_jb_2 -
(T_jb_2)_{B_\eps(}}dS \right)^2 \lesssim
\left(\frac{1}{|B_\eps|}\int\limits_{B_\eps} \norm{T_jb_2}
dS\right)^2\\
\lesssim
\frac{1}{|B_\eps|}\int\limits_{\dom} \norm{T_j b_2}^2 dS\lesssim
\frac{1}{|B_\eps|}\int\limits_{\dom} |b_2|^2 dS\\
 =\frac{1}{|B_\eps|}\int\limits_{B_{C\eps}}
|b(z)-b_{B_\eps}|^2 dS(z) \lesssim ||b||^2_{BMO}
\end{multline*}

Finally, estimating $T_j b_3$ we have
\begin{multline*}
\frac{1}{ |B_\eps |}\int\limits_{ B_\eps(z_0) } \norm{T_j b_3 -
(T_j b_3)_{ B_\eps }}dS\\ 
\lesssim
\frac{1}{ |B_\eps|^2 }\int\limits_{ B_\eps }\left(\int\limits_{ B_\eps }\norm{\int\limits_{ \dom }\left(K(\xi,w)-K(z,w)\right) b_3(w) dS(w)} dS(z)\right)dS(\xi)\\ 
\lesssim
 \frac{1}{ |B_\eps|^2 } \int\limits_{ B_\eps }\int\limits_{ B_\eps }\int\limits_{\dom}
 \frac{ d(\xi,z)^{1/2} }{ d(w,z)^{n+{1/2}} } |b_3(w)| dS(w)
 dS(z) dS(\xi)\\
  \lesssim \eps^{1/2} \int\limits_{\dom\setminus B_{ C\eps(z_0) }}
 \frac{ |b(w)-b_{C\eps(z_0)}| }{ d(w,z_0)^{n+{1/2}} }dS(w)\lesssim
 ||b||_{BMO},
\end{multline*}
 because $ d(\zeta,z)\lesssim \eps$ and $d(w,z)\asymp
 d(w,z_0)$ when $C>0$ is large enough.
The proof for $T'$ is analogous.
\qed}

By decomposition (\ref{eq:Luzin_decomposition}) this finalizes the proof of the theorem~\ref{thm:area_int_BMO}.

\end{document}